\documentclass[12pt, notitlepage]{amsart}
\usepackage{latexsym, amsfonts, amsmath, amssymb, amsthm, cite}

\newtheorem{lemma}{Lemma}
\newtheorem{proposition}{Proposition}
\newtheorem{theorem}{Theorem}

\newtheorem{corollary}{Corollary}

\usepackage{graphicx}
\usepackage{mathrsfs}

\title[Primes and divisor functions in arithmetic progressions]{Lower bounds for the variance of sequences in arithmetic progressions: primes and divisor functions}

\author{Adam J Harper}
\address{Jesus College \\ Cambridge \\ CB5 8BL \\ England} 
\email{\tt A.J.Harper@dpmms.cam.ac.uk}
\author{Kannan Soundararajan}
\address{Department of Mathematics\\ Stanford University\\ Stanford, CA 94305\\ USA} 
\email{\tt ksound@stanford.edu}
 \thanks{Adam Harper is supported by a research fellowship at Jesus College, Cambridge. Kannan Soundararajan was partially supported by NSF grant DMS 1500237, and a Simons Investigator grant from the Simons Foundation. The research for this paper was started when the first named author visited the second named author in April--May 2015, and he would like to thank Stanford University for their hospitality during this visit.}

\date{\today }

\begin{document}

\maketitle

\begin{abstract}
We develop a general method for lower bounding the variance of sequences in arithmetic progressions mod $q$, summed over all $q \leq Q$, building on previous work of Liu, Perelli, Hooley, and others. The proofs lower bound the variance by the minor arc contribution in the circle method, which we lower bound by comparing with suitable auxiliary exponential sums that are easier to understand.

As an application, we prove a lower bound of $(1-\epsilon) QN\log(Q^2/N)$ for the variance of the von Mangoldt function $(\Lambda(n))_{n=1}^{N}$, on the range $\sqrt{N} (\log N)^C \leq Q \leq N$. Previously such a result was only available assuming the Riemann Hypothesis.  
We also prove a lower bound $\gg_{k,\delta} Q N (\log N)^{k^2 - 1}$ for the variance of the divisor functions $d_k(n)$, valid on the range $N^{1/2+\delta} \leq Q \leq N$, 
for any natural number $k \geq 2$.
\end{abstract}

\section{Introduction} 

\noindent Suppose we are given a sequence ${\mathcal A} = (a_n)_{n=1}^{N}$ which we expect to be evenly distributed in 
arithmetic progressions:  precisely, we expect that for an arithmetic progression $a\pmod q$ with $(a,q)= h$, we have 
$$ 
\sum_{\substack{n\le N \\ n\equiv a \pmod q}} a_n \approx \frac{1}{\phi(q/h)} \sum_{\substack{n\le N \\ (n,q)=h} } a_n. 
$$ 
For example the sequences $a_n = \Lambda(n)$ and $a_n= d_k(n)$ (the $k$-divisor function, for any natural number $k$) 
are expected to satisfy the above property in wide ranges of $q$ relative to $N$.  For such a sequence $a_n$, we study here the variance 
  \begin{equation} 
  \label{1.1} 
  V(q;{\mathcal A}) := \sum_{h|q } \sum_{\substack{a\pmod q\\ (a,q)=h}} \Big| \sum_{n\equiv a \pmod q} a_n - \frac{1}{\phi(q/h)} \sum_{(n,q)=h} a_n \Big|^2,  
  \end{equation} 
  and more specifically the quantity  
 \begin{equation}
 \label{1.2}  
V({\mathcal A},Q) :=  \sum_{q\le Q} V(q;{\mathcal A}). 
 \end{equation}  
 In the range $N^{\frac 12} \leq Q \leq N$, we shall describe a general method to obtain a lower bound for the variance $V({\mathcal A},Q)$, and 
highlight the consequences for primes and divisor functions.

  \begin{theorem} 
  \label{primes}  Let $\epsilon >0$ be given, and let $N$ be large enough in terms of $\epsilon$, and let $Q$ be in the range $\sqrt{N}\le Q\le N$.  There exists an absolute constant $C$ such that  
  $$ 
  \sum_{q\le Q} \sum_{\substack{ a\pmod q \\ (a,q)=1}} \Big( \psi(N;q,a) - \frac{\psi_q(N)}{\phi(q)} \Big)^2 
  \ge (1-\epsilon) QN \Big(  \log \frac{Q^2}{N} - C \log \log N \Big) ,
  $$ 
  where
  $$ 
  \psi(N;q,a) := \sum_{\substack{n\le N \\ n\equiv a\pmod q} } \Lambda(n), \qquad \text{and} \qquad 
  \psi_q(N) := \sum_{\substack{n\le N\\ (n,q)=1}} \Lambda(n) .
  $$ 
  \end{theorem} 
  
  Theorem \ref{primes} continues a long line of investigations connected with the Barban--Davenport--Halberstam theorem, which established an upper bound $O(QN\log N)$ for the variance in Theorem  
  \ref{primes} in the range $N\ge Q \ge N (\log N)^{-A}$ for any fixed $A$.  The upper bound was refined by Montgomery~\cite{Montgomery} to the asymptotic $\sim QN\log Q$ in the same range of $Q$, and, on the assumption of GRH, Hooley~\cite{Hooley2} established this asymptotic in the range $Q \ge N^{\frac 12+\epsilon}$.   In \cite{FrGo1, FrGo2}, Friedlander and Goldston established bounds on the variance for individual $q$,
  $$
  \sum_{\substack{ a\pmod q \\  (a,q) =1}} \Big( \psi(N;q,a) - \frac{\psi_q(N)}{\phi(q) }\Big)^2,  
  $$ 
in a limited range for $q$ unconditionally, and in a wider range $N^{2/3+ \epsilon} \le q\le N$ conditional on the Generalized Riemann Hypothesis.  
Unconditional (weaker) lower bounds for the variance in Theorem \ref{primes} in restricted ranges were obtained by Liu \cite{Liu0, Liu} and Perelli \cite{Perelli}, with further refinements by Hooley \cite{Hooley13, Hooley15, Hooley16}.  In particular, the work of Perelli and the later papers~\cite{Hooley15, Hooley16} of Hooley allow ranges of $Q$ of the form $x^{1-c} \leq Q \leq x$, for some small fixed $c > 0$. Building on ideas of Friedlander and Goldston~\cite{FrGo1}, Hooley~\cite{Hooley13} also gave another conditional proof of the lower bound in Theorem \ref{primes}, now requiring only that the Riemann zeta function have no zeros with real part $> 3/4$. We should also mention that some of the previous literature concentrates, not on the true variance in Theorem \ref{primes}, but on the larger quantity
$$ 
\sum_{q\le Q} \sum_{\substack{ a\pmod q \\ (a,q)=1}} \Big( \psi(N;q,a) - \frac{N}{\phi(q)} \Big)^2 .
$$ 
Hooley~\cite{Hooley15} gave an unconditional lower bound for this quantity that is more or less the same as in Theorem \ref{primes}, essentially by exploiting the fact that if the Riemann zeta function did have zeros with large real part, these would give an additional positive contribution because of the difference between $N/\phi(q)$ and the true average $\psi_q(N)/\phi(q)$.
  
  \begin{theorem} 
  \label{kdiv}    Let $k \geq 2$ be a natural number.  Define 
  $$ 
  V_k(q)=V_k(q;N) := \sum_{h|q} \sum_{\substack{a\pmod q\\ (a,q)=h}} \Big ( \sum_{\substack{ n\le N\\ n\equiv a \pmod q}} d_k(n) - 
  \frac{1}{\phi(q/h)} \sum_{\substack{n\le N\\ (n,q)=h}} d_k(n) \Big)^2. 
  $$ 
  Let $\delta >0$ be a real number.  If $N$ is large enough in terms of $\delta$, then uniformly in the range $N^{\frac 12+\delta} \le Q \le N$ we have 
  $$ 
  \sum_{q\le Q} V_k( q) \gg_{k,\delta} QN (\log N)^{k^2-1}. 
  $$ 
  \end{theorem}

  The equidistribution of divisor functions in arithmetic progressions has been extensively studied; for example, in the case of $d_3(n)$ 
  there has been important progress on obtaining equidistribution in an individual arithmetic progression for large moduli, see \cite{FI, HB, FKM}, and for 
  the standard divisor function $d_2(n)$ such results were obtained in unpublished work of Hooley and Selberg, using Weil's bound for Kloosterman sums.   The variance of the divisor function $d_2$ in 
  arithmetic progressions $\pmod q$ has been studied by Motohashi \cite{Mot}, Blomer \cite{Blo}, and Lau and Zhao \cite{LZ}, and in \cite{LZ} an asymptotic for this variance is 
  obtained for individual $q$ with $N^{\frac 12+\epsilon} <q \le N^{1-\epsilon}$.    For larger $k$, the finer study of the variance of the $k$-divisor function in short intervals and arithmetic progressions has recently been initiated by Keating, Rodgers, Roditty-Gershon and Rudnick~\cite{KRRR}.   In particular, the work of \cite{KRRR} suggests the conjecture that 
  $$ 
  \sum_{q\le Q} V_k(q) \sim a_k QN(\log Q)^{k^2 -1} \gamma_k\Big(\frac{\log N}{\log Q}\Big),
  $$ 
  for a suitable positive constant $a_k$, and a complicated ``piecewise-polynomial" function $\gamma_k(x)$:  for each interval $x \in [\ell,\ell+1)$ (with $\ell=0$, $\ldots$, $k-1$) 
  the function $\gamma_k(x)$ is given by a polynomial in $x$ of degree $k^2-1$.   We remark that a closely related piecewise polynomial arose in the work of 
  Conrey and Gonek \cite{ConreyGonek} when they were formulating conjectures for the eighth moment of $\zeta(\frac 12+it)$.  
  Work in progress of Rodgers and the second author \cite{RodSound} establishes 
  a version of this conjecture when $k=3$ and $N^{\frac 12+\epsilon} \le Q \le N^{1-\epsilon}$, and for larger $k$ in a narrow range of values of $Q$ sufficiently close to $N$.  
  Theorem \ref{kdiv} adds to this literature by obtaining a lower bound of the right order of magnitude in the range $N^{\frac 12+\delta} <Q \leq N$; in view of the 
  results mentioned above, Theorem \ref{kdiv} is of interest for $k\ge 4$.  For other recent results related to the distribution of divisor functions (and other related functions like Hecke eigenvalues) in short intervals and  progressions,  see \cite{KowalskiRicotta, LesterYesha, FGKM, Lester}.

  We now outline the proofs of our theorems, starting with a general sequence ${\mathcal A}$ as in 
  \eqref{1.1} and \eqref{1.2}.     Define the associated exponential sum 
   \begin{equation} 
   \label{1.3} 
   {\mathcal A}(\alpha) := \sum_{n \leq N} a_n e(n\alpha) ,  
   \end{equation}
where as usual $e(\theta) := e^{2\pi i\theta}$. We also recall that the Ramanujan sum is given by 
   \begin{equation} 
   \label{1.4} 
   c_q(n) = \sum_{(a,q)=1} e(an/q) = \frac{\mu(q/(q,n)) \phi(q)}{\phi(q/(q,n))}. 
   \end{equation}  
We will first establish a general inequality connecting the variance $V({\mathcal A},Q)$ with the integral over ``minor arcs" of $|{\mathcal A}(\alpha)|^2$.  

\begin{proposition} \label{Prop1.3} Let $N$ be large, let $K\ge 5$ be a parameter, and let $Q_0$ and $K \sqrt{N\log N} \le Q \le N$ be such that 
\begin{equation} 
\label{1.5} 
\frac{N \log N}{Q} \le Q_0 \le \frac{Q}{K^2}. 
\end{equation} 
Let ${\frak M} = {\frak M}(Q_0,Q;K)$ denote the major arcs, consisting of those $\alpha \in {\Bbb R}/{\Bbb Z}$ having an approximation $|\alpha-a/q| \le K/(qQ)$ with $q\le K Q_0$ and $(a,q)=1$.  Let ${\frak m}$, the minor arcs, denote the complement of the major arcs in ${\Bbb R}/{\Bbb Z}$.  Then 
\begin{align*}
\sum_{Q_0 <q \le Q} V(q;{\mathcal A})& \ge Q\Big(1+ O\Big(\frac{\log K}{K}\Big)\Big) \int_{\frak m} |{\mathcal A}(\alpha)|^2 + O \Big( \frac{NK}{Q_0} \sum_{n\le N} |a_n|^2 \Big) \\
&-\sum_{q\le Q } \frac{1}{q} \sum_{\substack {d|q \\ d>Q_0}} \frac{1}{\phi(d)} 
\Big| \sum_{n} a_n c_d(n)\Big|^2. 
\end{align*} 
\end{proposition} 

Proposition \ref{Prop1.3}, and especially Proposition \ref{Propos1} below which forms the main step in its proof, generalises and simplifies the argument in section 4 of Hooley~\cite{Hooley15}. The idea of a connection between the variance in arithmetic progressions and the minor arc contribution in the circle method is widespread, and as Hooley notes both Liu~\cite{Liu0, Liu} and Perelli~\cite{Perelli} used it as well. However, the latter arguments relied on the connection between character sums and exponential sums (similarly as in the usual deductions of the multiplicative large sieve inequality), which can only be made to work (straightforwardly) when $a_n = \Lambda(n)$ or for other sequences without small prime factors. In contrast, the proof of Proposition \ref{Prop1.3} avoids Dirichlet characters and develops Hooley's approach, connecting the variance of $(a_n)_{n=1}^{N}$ in arithmetic progressions with the variance of the exponential sums $\mathcal{A}(a/q)$. By positivity of the variance one can discard the major arc contribution to the latter (which we would anyway probably expect to be small), leaving only a minor arc contribution and some terms involving Ramanujan sums $c_d(n)$ with $d$ fairly large.

For sequences such as the primes and divisor functions, the contribution of the sums involving the Ramanujan sum in Proposition \ref{Prop1.3} may be shown to be negligible, and it then remains to bound from below the minor arc contribution.   To do this, our idea is to introduce another sequence $({\tilde a}_n)_{n=1}^{N}$ that suitably approximates $a_n$, and such that 
  the associated exponential sum 
 \begin{equation} 
 \label{1.6} 
  {\tilde {\mathcal A}}(\alpha) = \sum_{n \leq N} {\tilde a}_n e(n\alpha) 
  \end{equation}  
  is more easily understood.    Then by Cauchy--Schwarz we have 
  \begin{equation} 
  \label{1.7} 
  \int_{\frak m} |{\mathcal A}(\alpha)|^2 d\alpha \ge \Big| \int_{\frak m} {\mathcal A}(\alpha)\overline{ \tilde {\mathcal A} (\alpha)} d\alpha \Big|^2 \Big( \int_{\frak m} |\tilde {\mathcal A} (\alpha)|^2 d\alpha \Big)^{-1} .
  \end{equation} 
  Since $\int_{\frak m} = \int_0^1 - \int_{\frak M}$, by Parseval's identity we get
  \begin{equation} 
 \label{1.8} 
 \int_{\frak m} {\mathcal A}(\alpha)\overline{ \tilde {\mathcal A} (\alpha)} d\alpha = \sum_n a_n \overline{{\tilde a}_n} - \int_{\frak M} 
 {\mathcal  A}(\alpha)\overline{ \tilde {\mathcal A} (\alpha)} d\alpha ,  
 \end{equation}
 and 
 \begin{equation}
 \label{1.9}  
 \int_{\frak m} |\tilde {\mathcal A} (\alpha)|^2 d\alpha = \sum_n |{\tilde a}_n|^2 - \int_{\frak M} |\tilde {\mathcal A} (\alpha)|^2 d\alpha. 
 \end{equation} 
  These observations reduce our problem to evaluating integrals over the major arcs.   
  
   To proceed further, we must specify more precisely the auxiliary sequence ${\tilde a}_n$.  Below it is convenient to pick a smooth function $\Phi$, compactly supported in $[0,1]$ with $0\le \Phi(t) \le 1$ for all $0\le t\le 1$ and with $\int_0^1 \Phi(t) dt \ge 1-\epsilon$, for some small $\epsilon > 0$.  Thus $\Phi$ may be viewed as a smooth approximation (from below) to the indicator function of the interval $[0,1]$.  We may clearly  choose $\Phi$ in such a way that for any $A>0$ we have 
   \begin{equation} 
 \label{1.10} 
 |{\hat \Phi}(\xi)| \ll_{\epsilon,A} (1+|\xi|)^{-A} ,
 \end{equation} 
where ${\hat \Phi}(\xi) = \int_{-\infty}^{\infty} \Phi(t) e(-\xi t) dt$ denotes the Fourier transform.   For such a choice of $\Phi$, we take 
\begin{equation} 
\label{1.11} 
{\tilde a}_n := \sum_{\substack{ r|n \\ r\le R}  } b_r \Phi(n/N),  
\end{equation} 
for a suitable choice of $b_r$, and $R$.  Let 
\begin{equation} 
\label{1.12}
B := \max_{r\le R} |b_r|. 
\end{equation} 
The motivation for the above construction is that on the major arc around $a/q$, one expects the behaviour of an exponential sum to be dictated by the distribution of the coefficients mod $q$. If the coefficients are a short divisor sum then one only has to understand the distribution of integers in intervals mod $q$. The presence of the smoothing $\Phi$ further helps to kill off all error terms, and ultimately to increase the permitted range of $Q$ in our Theorems.

\begin{proposition}  \label{prop major}  Keep notations as above, and assume that $KQ_0<R\le Q/(2K)$.  Then  
\begin{align*}
\int_{{\frak M}} {\mathcal A}(\alpha) \overline{\tilde {\mathcal A}(\alpha)} d\alpha &= N\sum_{q\le KQ_0} \int_{-\frac{K}{qQ}}^{\frac{K}{qQ}} 
\Big(\sum_{n\le N} a_nc_q(n) e(n\beta) \Big) \Big( \sum_{\substack{r\le R\\ q|r} } \frac{\overline{b_r}}{r} \Big) {\hat \Phi}(\beta N) d\beta \\
&\hskip .5 in + O\Big(\frac{K\sqrt{Q_0}}{\sqrt{Q}} BR (\log N) \Big(\sum_{n\le N} |a_n|^2 \Big)^{\frac 12} \Big) , 
\end{align*} 
and 
\begin{align*} 
\int_{\frak M} |\tilde {\mathcal A}(\alpha)|^2 d\alpha &=N\sum_{q\le KQ_0} \phi(q) \Big| \sum_{\substack {r\le R \\ q|r}} \frac{b_r}{r} \Big|^2  \Big(\int_{0}^{1} \Phi(t)^2 dt + O\left(\min\Big(1,\frac{qQ}{KN}\Big) \right) \Big) \\
&\hskip .5 in +O(B^2 R K Q_0 (\log N)^2) + O\Big(\frac{K\sqrt{Q_0}}{\sqrt{Q}}BR (\log N) \Big(\sum_{n\le N} |{\tilde a}_n|^2 \Big)^{\frac 12} \Big).
\end{align*} 
\end{proposition}

It is the introduction of the auxiliary sums $\tilde{\mathcal{A}}(\alpha)$ that allows us to obtain a wide range of $Q$ in Theorems \ref{primes} and \ref{kdiv}. In the previous literature the arguments proceeded directly with $\mathcal{A}(\alpha)$ (although they sometimes introduced auxiliary functions like $\tilde{a}_n$ in other contexts), which required a much more involved analysis and limited the range of $Q$. In the case of the primes, for example, previous arguments could involve information about zeros of $L$-functions at the depth of the log-free zero-density arguments of Linnik and Gallagher.

Given Proposition \ref{prop major}, the deduction of our theorem about primes is relatively straightforward because $(n,q)=1$ for almost all prime (or prime power) values of $n$, so the Ramanujan sum $c_q(n)$ takes the value $\mu(q)$ for almost all such values. Performing the calculations to deduce Theorem \ref{kdiv} is much less straightforward, but we carry this out fully in Section 8, particularly Section 8.1.    

Given the difficulty of this situation, we also provide an alternative approach to bounding the minor arc contribution in Proposition \ref{Prop1.3}, and complete 
the proof of Theorem \ref{kdiv} using this approach in Section 8.2.   Note that the Cauchy--Schwarz inequality \eqref{1.7} really gives 
\begin{equation} 
\label{1.13} 
\int_{\frak m}  |{\mathcal A}(\alpha)|^2 d\alpha \ge \Big( \int_{\frak m} |{\mathcal A}(\alpha) {\tilde {\mathcal A}}(\alpha)| d\alpha\Big)^2 \Big( \int_{\frak m} |{\tilde{\mathcal A}}(\alpha)|^2 d\alpha \Big)^{-1},  
\end{equation}  
so that one really needs only a lower bound for $\int_{\frak m} |{\mathcal A}(\alpha) \tilde{\mathcal A}(\alpha)| d\alpha$, which is potentially a simpler problem thanks to the absolute values on the inside. 

\begin{proposition}\label{newprop} Keep notations as above, and assume that $KQ_0 \le R\le \sqrt{N}$.   Suppose now that $|a_n| \ll_{\epsilon} N^{\epsilon}$ for any $\epsilon > 0$.  Then 
$$ 
\int_{\frak m} |{\mathcal A}(\alpha) {\tilde {\mathcal A}}(\alpha) | d\alpha \ge \sum_{KQ_0 < q\le R} \Big| \sum_{\substack{r\le R \\ q|r}} \frac{b_r}{r} \Big| \Big| \sum_{n\le N} 
a_n c_q(n) \Phi\Big(\frac{n}{N}\Big)\Big|    + O_{\epsilon}(B RN^{\frac 12 +\epsilon}  ). 
$$
\end{proposition}

  \section{Connecting the variance to exponential sums}  
  
  \noindent 
 To study the variance $V(q;{\mathcal A})$ of the sequence $\mathcal{A}$, it turns out to be helpful to consider the variance of the exponential sums ${\mathcal A}(a/q)$ over all reduced residue classes $a\pmod q$.  Since $\sum_{(a,q)=1} {\mathcal A}(a/q) = \sum_{n\le N} a_n c_q(n)$, we may define this  variance by setting
  \begin{equation} 
  \label{2.1} 
  H(q;{\mathcal A}) : = \sum_{(a,q)=1} \Big| {\mathcal A}(a/q) - \frac{1}{\phi(q) } \sum_{n} a_n c_q(n) \Big|^2 = 
  \sum_{(a,q)=1} |{\mathcal A}(a/q)|^2 - \frac{1}{\phi(q) } \Big| \sum_n a_n c_q(n) \Big|^2. 
  \end{equation} 
  
   \begin{proposition} \label{Propos1} With the above notation we have 
  $$ 
  H(q;{\mathcal A}) = \sum_{d|q} dV(d;{\mathcal A}) \mu(q/d), 
  $$ 
  or equivalently 
  $$ 
  qV(q;{\mathcal A})  = \sum_{d|q} H(d;{\mathcal A}). 
  $$
  \end{proposition}  
  
  The key to the proof of this proposition is the following identity for Ramanujan sums. 
  
  \begin{lemma}  For any two integers $m$ and $n$ we have 
  $$ 
  \sum_{d|q} \frac{1}{\phi(d)} c_d(m) c_d(n) = \begin{cases} 0 &\text{if  } (m,q) \neq (n,q)\\ 
  q/\phi(q/h) &\text{if } (m,q) = (n,q) = h. \\
  \end{cases} 
  $$ 
  \end{lemma}  
   \begin{proof}  Both sides of the claimed identity are multiplicative functions of $q$ (for fixed $m$ and $n$).  Thus it suffices to check the identity at prime powers $q=p^k$.   Assume without loss of generality that $(m,p^k) =p^a$ and $(n,p^k) =p^b$ with $k\ge a \ge b$.  If $a>b$ then the left hand side is 
   $$ 
   \sum_{\ell=0}^{b} \frac{1}{\phi(p^\ell)} \phi(p^\ell)^2 + \frac{1}{\phi(p^{b+1})} \phi(p^{b+1}) \frac{\mu(p)\phi(p^{b+1})}{\phi(p)} = p^b - p^b=0 ,
   $$
as required. If now $a=b$ and $k=a$ we get 
   $$ 
   \sum_{\ell =0}^{k} \frac{1}{\phi(p^\ell)} \phi(p^\ell)^2 = p^k, 
   $$ 
   again matching the right hand side.  Finally if $a=b$ and $k>a$ then we get 
   $$ 
   \sum_{\ell=0}^{a} \frac{1}{\phi(p^\ell)} \phi(p^\ell)^2 + \frac{1}{\phi(p^{a+1})} \frac{\mu(p)^2 \phi(p^{a+1})^2}{\phi(p)^2} = p^a + \frac{p^{a}}{(p-1)} = \frac{p^{a+1}}{p-1},  
   $$ 
   which again matches our right hand side.  
   \end{proof}  
   
\begin{proof}[Proof of Proposition \ref{Propos1}]     We prove the second of the two equivalent formulae stated there.  
   First we expand out the inner sum in \eqref{1.1}  to obtain 
   \begin{align*}
   \sum_{\substack {a\pmod q \\ (a,q)=h}}\Big( \sum_{m,n \equiv a\pmod q} a_m \overline{a_n} -& 2 
     \sum_{n\equiv a\pmod q} a_n \Big(\frac{1}{\phi(q/h)} \sum_{(n,q)=h} \overline{a_n} \Big) + 
     \frac{1}{\phi(q/h)^2} \Big|\sum_{(n,q)=h} a_n \Big|^2 \Big) \\
     =& \sum_{\substack { m\equiv n\pmod q \\ (n,q)=h} } a_m \overline{a_n} -\frac{1}{\phi(q/h)} \Big|\sum_{(n,q)=h} a_n \Big|^2. 
     \end{align*} 
     Summing this over $h|q$, we find that 
    \begin{equation} 
    \label{Prop1} 
    qV(q;{\mathcal A}) = q \sum_{m\equiv n\pmod q} a_m \overline{a_n}  - q \sum_{h|q} \frac{1}{\phi(q/h)} \Big|\sum_{(n,q)=h} a_n \Big|^2. 
    \end{equation} 
    
    On the other hand note that, by the definition \eqref{2.1},
    $$ 
    H(d;{\mathcal A}) = 
    \sum_{(a,d) =1} \sum_{m, n } a_m \overline{a_n} e(a(m-n)/d) - \frac{1}{\phi(d)} \Big|\sum_{n} a_n c_d(n)\Big|^2.
    $$ 
    Sum this over all divisors $d$ of $q$.  The first term above contributes 
    $$ 
    \sum_{d|q} \sum_{m,n} a_m \overline{a_n} c_d(m-n) = q\sum_{m\equiv n \pmod q} a_m \overline{a_n},
    $$ 
    which matches the first term in the right hand side of \eqref{Prop1}.  Appealing to Lemma 1, the second term above 
    contributes 
    $$ 
    - \sum_{d|q} \frac{1}{\phi(d)}  \sum_{m,n} a_m \overline{a_n} c_d(m) c_d(n) = - \sum_{\substack{m, n\\ (m,q)=(n,q)}} a_m \overline{a_n} \frac{q}{\phi(q/(q,m))},
    $$ 
    matching the second term in the right hand side of \eqref{Prop1}.  This completes the proof.
\end{proof}
    
    We shall actually make use of the following corollary to Proposition \ref{Propos1}, which follows upon noting that $H(d;{\mathcal A})$ is non-negative. 
    
    \begin{corollary} \label{Cor1}  For any parameter $Q_0$, we have 
    $$ 
q    V(q;{\mathcal A} )  \ge \sum_{\substack{a\pmod q \\ q/(a,q) >Q_0}} |{\mathcal A}(a/q)|^2 - \sum_{\substack{d|q \\ d>Q_0}} 
\frac{1}{\phi(d)} \Big|\sum_{n} a_n c_d(n)\Big|^2. 
$$ 
\end{corollary} 

\section{Bounding exponential sums by minor arcs: Proof of Proposition \ref{Prop1.3}}  

\noindent Throughout this section we keep in mind the notation in Proposition \ref{Prop1.3}.  Thus recall that $K\ge 5$, that $K \sqrt{N\log N}\le Q\le N$ and  that $N(\log N)/Q \le Q_0 \le Q/K^2$, and recall also the definitions of the major arcs ${\frak M}$ and minor arcs ${\frak m}$.  We begin with a general lemma on Diophantine approximation. 

\begin{lemma} \label{diop}   Define 
$$ 
f(\alpha) := \sum_{Q_0<q\le Q} \frac{1}{q} \sum_{\substack{a\pmod q \\ q/(a,q)>Q_0 \\  |\alpha -a/q| \le K/(Q_0Q)} } 1.
$$ 
If $|\alpha- a_0/q_0| \le K/(q_0 Q)$ for some $KQ_0 \le q_0 \le Q/K$ and $(a_0,q_0)=1$, then 
$$
f(\alpha) \ge \frac{2K}{Q_0}\Big( 1- \frac 5K - \frac{\log K}{K}\Big). 
$$ 
In particular, the lower bound above for $f(\alpha)$ holds for all $\alpha \in {\frak m}$.  
\end{lemma}
\begin{proof}   Suppose that $|\alpha -a_0/q_0| \le K/(q_0Q)$ with $(a_0,q_0)=1$ and $KQ_0 \le q_0 \le Q/K$. We will construct pairs $(a,q)$ that give a contribution to the sums in $f(\alpha)$.   Let $q= \ell q_0 + b$ where $0\le b <q_0$ and $Q/(K q_0) \le \ell \le (Q/q_0-1)$.   Consider only those values $b$ such that $\Vert ba_0/q_0 \Vert \le \ell q_0 (K-1)/(Q_0 Q)$.  For a given $\ell$, note that the number of permitted choices for $b$ is at least 
$$ 
2 \ell q_0^2 \frac{K-1}{Q_0Q} - 1 \ge 2\ell q_0^2 \frac{(K-2)}{Q_0Q}.
$$ 
Given such a choice of $b$ and $q$, select $a$ such that $|a/q- a_0/q_0| = \frac{1}{q}|a - a_{0}l - \frac{a_0 b_0}{q_0}| = \frac{1}{q} \Vert b a_0/q_0\Vert $.

Then note that 
$$ 
\Big| \alpha - \frac{a}{q} \Big| \le \Big| \frac{a_0}{q_0} - \frac{a}{q} \Big| + \frac{K}{q_0Q} \le \frac{\Vert ba_0/q_0 \Vert}{\ell q_0} + \frac{K}{q_0Q} \le \frac{K-1}{Q_0Q} + \frac{K}{q_0Q} \le \frac{K}{Q_0Q}. 
$$ 
 Moreover, for such a choice of $b$ (and hence for $a$), if we write $a/q =a^{\prime}/q^{\prime}$ with $(a^{\prime},q^{\prime})=1$ then, if we don't already have $a^{\prime}/q^{\prime} = a_{0}/q_{0}$, we have
$$ 
\frac{1}{q_0q^{\prime}} \le \Big| \frac{a_0}{q_0} - \frac{a^{\prime}}{q^{\prime}} \Big| = \Big| \frac{a_0}{q_0} - \frac{a}{q} \Big| = \frac{1}{q} \Vert ba_0/q_0 \Vert \le \frac{K-1}{Q_0Q},
$$ 
and it follows that $q^{\prime} \ge Q_0Q /((K-1)q_0) \ge Q_0$.  Thus $a/q$ is an admissible fraction counted in the definition of $f(\alpha)$.

  Therefore 
$$ 
f(\alpha) \ge \sum_{Q/(Kq_0) \le \ell \le (Q/q_0 -1) } \frac{1}{(\ell +1) q_0} \Big( 2 \ell q_0^2 \frac{K-2}{Q_0 Q} \Big) 
\ge 2 q_0 \frac{K-2}{Q_0Q} \Big( \frac{Q}{q_0} - \frac{Q}{Kq_0} -2 - \log K \Big), 
$$ 
and the stated lower bound follows upon noting that $Q/q_0 \ge K$.  

Finally, note that every $\alpha$ has a Diophantine approximation $|\alpha -a_0/q_0|  \le K/(q_0Q)$ with $q_0 \le Q/K$ and $(a_0,q_0)=1$, and if $\alpha \in {\frak m}$ then by definition we must have $q_0 > KQ_0$ and so the bound just derived applies.
\end{proof} 


We now turn to the proof of Proposition \ref{Prop1.3}.  Applying Corollary \ref{Cor1} to lower bound all the terms $V(q;\mathcal{A})$, we see that it is enough to establish that 
\begin{equation} 
\label{3.0}
\sum_{Q_0 < q \le Q} \frac{1}{q} \sum_{\substack{ a\pmod q \\ q/(a,q)>Q_0}} |{\mathcal A}(a/q)|^2 \ge Q \Big(1-\frac{5+\log K}{K}\Big) \int_{{\frak m}} |{\mathcal A}(\alpha)|^2 d\alpha + O\Big( \frac{NK}{Q_0} \sum_{n \le N} |a_n|^2 \Big).
\end{equation}

Let $f(\alpha)$ be defined as in Lemma \ref{diop}, so that
\begin{align} 
\label{3.1}
 \frac{2K}{Q_0} \Big(1-\frac 5K -\frac{\log K}{K} \Big) 
\int_{{\frak m}} |{\mathcal A}(\alpha)|^2 d\alpha &\le \int_0^1 f(\alpha) |{\mathcal A}(\alpha)|^2 d\alpha \nonumber \\ 
&= \sum_{Q_0 < q \le Q} \frac{1}{q} \sum_{\substack { a\pmod q \\ q/(a,q) > Q_0} } \int_{-\frac{K}{Q_0Q}}^{\frac{K}{Q_0Q}} 
\Big|{\mathcal A}\Big(\frac aq + \beta \Big) \Big|^2 d\beta. 
\end{align}  
Now note that 
$$
|{\mathcal A}(a/q+\beta)|^2 = |{\mathcal A}(a/q)|^2 + O\Big((|{\mathcal A}(a/q)|+|{\mathcal A}(a/q+\beta)|) |{\mathcal A}(a/q+\beta)-{\mathcal A}(a/q)|\Big),
$$ 
and so the quantity in \eqref{3.1} equals 
\begin{equation} 
\label{3.2} 
\frac{2K}{Q_0Q} \sum_{Q_0 < q\le Q} \frac{1}{q} \sum_{\substack { a\pmod q\\ q/(a,q) >Q_0}} |{\mathcal A}(a/q)|^2 + E, 
\end{equation} 
say, where, by using Cauchy--Schwarz, $E \ll \sqrt{E_1 E_2}$ with 
$$ 
E_1 = \int_{-\frac{K}{Q_0Q}}^{\frac{K}{Q_0Q}} \sum_{Q_0 < q \le Q} \frac 1q \sum_{\substack { a\pmod q\\ q/(a,q) >Q_0}} \Big( |{\mathcal A}(a/q)|^2 + |{\mathcal A}(a/q+\beta)|^2 \Big) d\beta, 
$$ 
and 
$$ 
E_2 = \int_{-\frac{K}{Q_0Q}}^{\frac{K}{Q_0Q}} \sum_{Q_0 < q \le Q} \frac 1q \sum_{\substack { a\pmod q\\ q/(a,q) >Q_0}}  \Big| 
{\mathcal A}(a/q+\beta) -{\mathcal A}(a/q)\Big|^2 d\beta. 
$$ 

We now use the large sieve (see for example Chapter 27 of \cite{Dav}) to bound $E_1$ and $E_2$ (this being the standard approach for comparing the sum of $|\mathcal{A}(\cdot)|^2$ at discrete points with the integral around the whole circle).  Write $a/q$ as a reduced fraction $b/r$.   Since $r=q/(a,q)$ we then 
have $Q_0 < r \le Q$, and for each such $r$ note that $\sum_{Q_0 <q \le Q, r|q} 1/q \ll (1 +\log (Q/r))/r$.   Thus 
$$ 
E_1 \ll \int_{-\frac{K}{Q_0Q}}^{\frac{K}{Q_0Q}}  \sum_{Q_0 < r \le Q} \frac{1}{r} \Big(\log \frac{Q}{r} + 1 \Big) \sum_{(b,r)=1} \Big( |{\mathcal A}(b/r+\beta)|^2 + |{\mathcal A}(b/r)|^2\Big) d\beta,  
$$
and splitting the sum over $r$ into dyadic intervals and using the large sieve, we obtain that 
$$ 
E_1 \ll \frac{K}{Q_0Q} \Big( \frac{N}{Q_0} \Big(1+\log \frac{Q}{Q_0} \Big) + Q \Big) \sum_n |a_n|^2 \ll 
\frac{K}{Q_0} \sum_{n} |a_n|^2. 
$$ 
By writing $\mathcal{A}(a/q + \beta) - \mathcal{A}(a/q) = \sum_{n \leq N} a_n (e(n\beta) - 1) e(na/q)$ and using the large sieve, we find that 
$$ 
E_2 \ll \frac{K}{Q_0} \max_{|\beta|\le K/(Q_0Q)} \sum_{n} |a_n|^2 |e(n\beta)-1|^2 \ll \frac{K}{Q_0} \Big(\frac{NK}{Q_0Q}\Big)^2 \sum_{n} |a_n|^2.
$$ 
Using these estimates in \eqref{3.1} and \eqref{3.2}, we obtain the desired estimate \eqref{3.0}, and 
thus Proposition \ref{Prop1.3} follows.
\qed

  \section{Evaluating exponential sums on major arcs: Proof of Proposition \ref{prop major}}  
  
  \noindent Throughout we keep in mind the notation of Proposition \ref{prop major}, and in particular 
  \eqref{1.6} through \eqref{1.12}.   
  
\begin{lemma} 
\label{major} Suppose that $\alpha =a/q + \beta$ with $|\beta|\le 1/(2qR)$, $q\le R$ and $(a,q)=1$.  
 Then we have 
$$ 
\tilde {\mathcal A}(\alpha) = N {\hat \Phi}(-N\beta) \sum_{\substack{r\le R \\ q|r} } \frac{b_r}{r}   
+O( BR\log N ). 
$$ 
\end{lemma} 

Note that the first term here is independent of the value of $a$.

\begin{proof}  Using the Poisson summation formula, we see that 
$$ 
\tilde {\mathcal A}(\alpha) = \sum_{r\le R} b_r \sum_{m} e(\alpha mr) \Phi\Big(\frac{mr}{N}\Big) = N 
\sum_{r\le R}\frac{b_r}{r} \sum_{k} {\hat \Phi}\Big(\frac{N}{r} (k-r\alpha)\Big). 
$$ 
Consider first the contribution of terms with $q\nmid r$.  If $k$ is the nearest integer to $r\alpha$ then $|k-r\alpha| = \Vert r\alpha \Vert \ge \Vert ra/q\Vert - \Vert r\beta \Vert \ge \Vert ra/q\Vert - 1/(2q) \ge \Vert ra/q\Vert/2$.  Therefore, using the decay bound \eqref{1.10} with $A=1$ (for the term closest to $r\alpha$) and $A=2$ (for all other terms), 
$$ 
\sum_{k} {\hat \Phi} \Big(\frac Nr (k-r\alpha)\Big) \ll \frac{r}{N\Vert ar/q\Vert} + \frac{r^2}{N^2} 
\ll \frac{r}{N\Vert ra/q\Vert}, 
$$ 
and so the total contribution of the terms with $q\nmid r$ is 
$$ 
\ll B \sum_{\substack { r\le R \\ q\nmid r }} \frac{1}{\Vert ra/q\Vert} \ll B \frac{R}{q} \sum_{1 \leq r \leq q-1} \frac{q}{r} \ll BR \log N. 
$$ 

Now consider the terms with $q|r$.  The nearest integer to $r\alpha$ is then $ra/q$, and so 
$$ 
\sum_{k} {\hat \Phi}\Big(\frac Nr (k-r\alpha) \Big)  = {\hat \Phi}(-N\beta) +O\Big( \sum_{k\neq ra/q} \Big(\frac{r}{N|k-r\alpha|}\Big)^2 \Big) 
= {\hat \Phi}(-N\beta) + O\Big( \frac{r^2}{N^2} \Big). 
$$ 
The lemma follows.  
 \end{proof} 

\begin{proof}[Proof of Proposition \ref{prop major}]   We begin with the first assertion.  
Let ${\frak M}(q)$ denote the union of the major arcs around $a/q$ for all $(a,q)=1$, where we assume now that $q\le KQ_0 \le R$.  Apply Lemma \ref{major} to $\overline {\tilde {\mathcal A}(\alpha)}$, and consider first the contribution of the main term there to the integral over ${\frak M}(q)$.  This equals 
\begin{align*}
&\sum_{(a,q)=1} \int_{-\frac{K}{qQ}}^{\frac{K}{qQ} }\Big( \sum_{n\le N} a_n e(na/q + n\beta) \Big) \Big( N{\hat \Phi}(\beta N) \sum_{\substack{r\le R \\ q|r}  } \frac{\overline{b_r}}{r} \Big) d\beta\\
 = &N \int_{-\frac{K}{qQ}}^{\frac{K}{qQ}} \Big(\sum_{n\le N} a_nc_q(n) e(n\beta) \Big) \Big( \sum_{\substack{r\le R\\ q|r} } \frac{\overline{b_r}}{r} \Big) {\hat \Phi}(\beta N) d\beta, 
\end{align*}
and summing this over $q\le KQ_0$ gives the main term of the Proposition.  

Now consider the contribution of the remainder term in Lemma \ref{major} to the integral over ${\frak M}$.  By Cauchy--Schwarz, and since $q\le KQ_0 \le R$, this is 
$$
\ll BR(\log N) \int_{{\frak M}}  |{\mathcal A}(\alpha)|  d\alpha  \ll BR (\log N)|{\frak M}|^{\frac 12} 
\Big(\int_0^1 |{\mathcal A}(\alpha)|^2 d\alpha\Big)^{\frac 12}, 
$$ 
where $|{\frak M}|$ denotes the measure of the major arcs, which is $\ll K^2 Q_0/Q$.  The first case now follows by Parseval.  

For our second integral, the same argument gives (with $\alpha=a/q+\beta$)
$$ 
\int_{\frak M} |\tilde {\mathcal A}(\alpha)|^2 d\alpha = N\int_{\frak M} \tilde {\mathcal A}(\alpha) \Big( \sum_{\substack{ r\le R \\ q|r}} \frac{\overline{b_r}}{r} \Big) 
{\hat \Phi}(\beta N) d\alpha + O\Big(\frac{K\sqrt{Q_0}}{\sqrt{Q}} BR (\log N) \Big(\sum_{n\le N} |{\tilde a}_n|^2 \Big)^{\frac 12} \Big) .  
$$
Now we use Lemma \ref{major} again to simplify the main term above.  The main term from Lemma \ref{major} leads to a term
$$ N\sum_{q\le KQ_0} \phi(q) \Big| \sum_{\substack {r\le R \\ q|r}} \frac{b_r}{r} \Big|^2  \Big(N \int_{-\frac{K}{qQ}}^{\frac{K}{qQ}} |\hat{\Phi}(\beta N)|^2 d\beta \Big) = N\sum_{q\le KQ_0} \phi(q) \Big| \sum_{\substack {r\le R \\ q|r}} \frac{b_r}{r} \Big|^2  \Big(\int_{-\frac{NK}{qQ}}^{\frac{NK}{qQ}} |\hat{\Phi}(u)|^2 du \Big) , $$
and Parseval's identity together with our decay estimate for $\hat{\Phi}$ show this is equal to
$$ N\sum_{q\le KQ_0} \phi(q) \Big| \sum_{\substack {r\le R \\ q|r}} \frac{b_r}{r} \Big|^2  \Big(\int_{0}^{1} \Phi(t)^2 dt + O\left(\min\Big(1,\frac{qQ}{KN}\Big) \right) \Big) , $$
as in the statement of the proposition. Again recalling our decay estimate for $\hat{\Phi}$, the error term from Lemma \ref{major} contributes 
\begin{align*}
&\ll NBR (\log N) \sum_{q\le KQ_0} \phi(q) \sum_{\substack{ r\le R\\ q|r} } \frac{|b_r|}{r} \int_{-\frac{K}{qQ}}^{\frac{K}{qQ}} 
|{\hat \Phi}(\beta N) | d\beta \\
&\ll B^2 R (\log N) \sum_{q\le KQ_0} \frac{\phi(q) \log N}{q} \ll B^2 R K Q_0 (\log N)^2,  
\end{align*} 
completing our proof.
\end{proof}
 
We end this section by casting $\sum_n a_n \overline{\tilde a_n}$ in \eqref{1.8} into a form similar to the 
main term of our first formula in Proposition \ref{prop major}.  This will be useful when executing one of our proofs of Theorem 2, see Section 8.1.


 \begin{lemma} \label{lem5}  With notations as above, 
 $$
 \sum_n a_n \overline{\tilde a_n} = \sum_{q\le R} \Big( \sum_{\substack{r\le R \\ q|r}} \frac{\overline{b_r}}{r} \Big) \sum_{n} a_n c_q(n) \Phi\Big(\frac nN\Big). 
 $$ 
  \end{lemma}  
 \begin{proof} Note that $\sum_{q|r} c_q(n)$ equals  $r$ if $r|n$, and $0$ if $r\nmid n$.  Therefore 
 $$ 
 \sum_n a_n \overline{\tilde a_n} = \sum_n a_n \Phi\Big(\frac nN\Big) \sum_{\substack{ r\le R \\ r|n}} \overline{b_r} 
 = \sum_n a_n \Phi\Big(\frac nN\Big) \sum_{\substack{r\le R }} \frac{\overline{b_r}}{r} \sum_{q|r} c_q(n),
 $$ 
 and the result follows upon rearranging sums.  
 \end{proof} 
 
 \section{Proof of Proposition \ref{newprop}}

 \noindent For  $KQ_0 < q\le R (\le \sqrt{N} \le Q/(2K))$ and $1\le a \le q-1$ with $(a,q)=1$, note that the intervals $(\frac{a}{q} - \frac{1}{2qR}, \frac{a}{q} + \frac{1}{2qR})$ are all disjoint, 
 and do not overlap with any major arc.   Thus these intervals are all contained in the minor arcs, and therefore
 \begin{equation} 
 \label{newsec1}
 \int_{\frak m} |{\mathcal A}(\alpha) \tilde{\mathcal A}(\alpha)| d\alpha 
 \ge \sum_{KQ_0 < q\le R} \sum_{\substack{ a\pmod q\\ (a,q)=1}} \int_{-\frac{1}{2qR}}^{\frac{1}{2qR}} |{\mathcal A}(a/q+ \beta) \overline{\tilde{\mathcal A}(a/q+\beta)}| d\beta. 
 \end{equation} 
  Now we use Lemma \ref{major} to evaluate $\overline{\tilde{\mathcal A}(a/q+\beta)}$.  The remainder term arising from that Lemma contributes, using Cauchy--Schwarz and Parseval,  
  $$ 
  \ll BR \log N \Big(\sum_{KQ_0 < q\le R} \sum_{\substack{ a\pmod q\\ (a,q)=1}} \int_{-\frac{1}{2qR}}^{\frac{1}{2qR}}  d\beta \Big)^{\frac 12} \Big( \int_0^1 |{\mathcal A}(\alpha)|^2 d\alpha\Big)^{\frac 12} \ll BR(\log N) \Big( \sum_{n\le N} |a_n|^2 \Big)^{\frac 12}. 
  $$ 
  Since $|a_n| \ll_{\epsilon} N^{\epsilon}$ by assumption, this is $\ll_{\epsilon} BRN^{\frac 12+\epsilon}$.  
  
  Using the triangle inequality, the main term from Lemma \ref{major} contributes to the right side of \eqref{newsec1} an amount  
 \begin{equation} 
 \label{newsec2} 
 \ge N \sum_{KQ_0 < q\le R} \Big| \sum_{\substack{r\le R \\ q|r}} \frac{\overline{b_r}}{r}  \int_{-\frac{K}{2qR}}^{\frac{K}{2qR}} {\hat {\Phi}}(N\beta) \sum_{n\le N} a_n e(n\beta) \sum_{\substack{a\pmod q\\ (a,q) =1}} e(an/q) d\beta \Big|. 
 \end{equation}  
 Now note that 
 $$ 
 \int_{-\frac{1}{2qR}}^{\frac{1}{2qR}} {\hat \Phi}(N\beta) e(n\beta) d\beta = \frac{1}{N} \int_{-\frac{N}{2qR}}^{\frac{N}{2qR}} {\hat \Phi}(u) e(nu/N) du = 
 \frac{1}{N}\Big( \Phi\Big(\frac{n}{N}\Big) + O\Big( \min \Big( 1, \frac{qR}{N}\Big)\Big)\Big). 
 $$ 
 The main term above, when inserted in \eqref{newsec2} leads to the main term of our proposition.   The remainder term above contributes to \eqref{newsec2} an amount 
 $$ 
 \ll \sum_{KQ_0 < q\le R} \frac{B\log N}{q} \min\Big(1,\frac{qR}{N}\Big) \sum_{n\le N} |a_n c_q(n)| = \frac{BR\log N}{N} \sum_{KQ_0 < q\le R} \sum_{n\le N} |a_n| |(q,n)| \ll_{\epsilon} BR^2 N^{ \epsilon} ,
 $$ 
 and so Proposition \ref{newprop} follows.

\section{The case of primes: Proof of Theorem 1} 
    
\noindent We apply our previous work taking $a_n = \Lambda(n)$ for $n\le N$.   We shall  take $K= (\log N)^2$, and $Q_0 = N (\log N)^{10}/Q$, and we shall also assume that $\sqrt{N} (\log N)^{100} \le Q \le N$.  
Put also $\psi(N;\alpha) := \sum_{n\le N} \Lambda(n)e(n\alpha)$.

{\em Preparation.} Note that in this setting, with $\psi_q(N) =\sum_{n\le N, (n,q)=1} \Lambda(n)$, the variance \eqref{1.1} in progressions $\pmod q$ is 
$$ 
V(q; \mathcal{A}) = \sum_{\substack{a\pmod q \\ (a,q)=1}} \Big(\psi(N;q,a) - \frac{\psi_q(N)}{\phi(q)} \Big)^2 + O((\log N)^2).
$$
In other words, it makes little difference if we only keep the term $h=1$ in the outer sum in \eqref{1.1}.

{\em Applying Proposition \ref{Prop1.3}.} Since, for $d\le N$,  
$$ 
\sum_{n\le N} \Lambda(n)c_d(n) \leq \sum_{n \leq N} \Lambda(n) (d,n) \ll N + \sum_{p^{k} || d} p^{k} \log N \ll N\log N , 
$$ 
we find that 
\begin{align*}
\sum_{q\le Q} \frac{1}{q} \sum_{\substack{d|q \\ d>Q_0} } \frac{1}{\phi(d)} \Big|\sum_{n\le N} \Lambda(n) c_d(n) \Big|^2 &\ll  \sum_{q \le Q} \frac{1}{q} \sum_{\substack{ d|q \\ d>Q_0} } \frac{N^2 (\log N)^2}{\phi(d)} \ll 
\sum_{Q_0< d \le Q} \frac{N^2 (\log N)^2}{\phi(d)} \frac{\log N}{d} \\
&\ll \frac{N^2 (\log N)^3 }{Q_0}\ll QN. 
\end{align*} 
Using the simple estimate $\sum_{n\le N} \Lambda(n)^2 \ll N\log N$, and appealing to Proposition \ref{Prop1.3},  we obtain 
 \begin{align} 
 \label{5.1}
 \sum_{Q_0 < q\le Q} \sum_{ \substack{a\pmod q \\ (a,q)=1}} \Big(\psi(N;q,a) - \frac{\psi_q(N)}{\phi(q)} \Big)^2  
& \ge Q \int_{\frak m} |\psi(N;\alpha)|^2d\alpha  +O(NQ). 
 \end{align} 
  
{\em Applying Proposition \ref{prop major}.} To estimate the integral over the minor arcs, we use our work leading up to Proposition \ref{prop major}.   
 With $R= Q/(\log N)^{20}$, we take the usual sieve-type weights
 $$ 
 b_r = \begin{cases} 
 \mu(r) \log (R/r) &\text{if  } r\le R \\ 
 0 &\text{if  } r>R. 
 \end{cases}
 $$ 
 Note that, in the notation of \eqref{1.12}, we have $B = \log R$.  
  Set 
  $$ 
  \tilde{\Lambda}(n) := \sum_{\substack{ r|n \\ r\le R}} b_r \Phi\Big(\frac{n}{N}\Big), \qquad \text{and } \qquad \tilde{\psi}(N;\alpha) := \sum_{n\le N} \tilde{\Lambda}(n) e(n\alpha).
  $$ 
  Note that, using the prime number theorem and summation/integration by parts, 
  \begin{align} 
  \label{5.2} 
  \sum_{n\le N} \Lambda(n) \tilde{\Lambda}(n)& =\sum_{p\le N} (\log p) (\log R) \Phi\Big(\frac{p}{N} \Big) + O(R\log R) + O(N^{\frac 12+\epsilon}) 
\nonumber\\
&  = N \log R \int_0^1 \Phi(t) dt + O(N). 
  \end{align}  
  Further, by the main Theorem of Graham \cite{Graham} and partial summation, we get 
  \begin{equation} 
  \label{5.3} 
  \sum_{n\le N} \tilde{\Lambda}(n)^2 = \sum_{n \leq N} \Big(\sum_{\substack{ r|n \\ r\le R}} \mu(r) \log(R/r) \Big)^2 \Phi\Big(\frac{n}{N} \Big)^2 = N \log R \int_0^1 \Phi(t)^2 dt + O(N) .
 \end{equation} 
We will also use the following asymptotic  (valid for $q\le R$, for any $\delta>0$, and with some $c>0$):
\begin{equation} 
\label{5.4} 
\sum_{\substack{r \le R \\ q|r}} \frac{b_r}{r} = \sum_{\substack{ r\le R \\ q|r }} \frac{\mu(r)}{r} \log \frac{R}{r} = \frac{\mu(q)}{\phi(q)} + O\Big(\frac{1}{q} \exp(-c\sqrt{\log R/q}) \prod_{p|q} \Big(1+O\Big(\frac{1}{p^{1-\delta}}\Big)\Big)\Big). 
\end{equation} 
This follows by a standard argument, writing (we may clearly assume that $q$ is square-free) 
\begin{align*}
\sum_{\substack {r\le R \\ q|r} } \frac{\mu(r)}{r} \log \frac{R}{r} &= \frac{\mu(q)}{q} \sum_{\substack{r \leq R/q\\ (r,q)=1}} \frac{\mu(r)}{r} \log(R/rq)\\
& = \frac{\mu(q)}{q} \frac{1}{2\pi i} \int_{c-i\infty}^{c+i\infty} \frac{1}{\zeta(s+1)} \prod_{p|q} \Big(1-\frac{1}{p^{s+1}}\Big)^{-1} \Big(\frac{R}{q}\Big)^s \frac{ds}{s^2}, 
\end{align*} 
and then shifting contours appropriately, staying within the classical zero-free region for $\zeta(s)$.  

We wish to evaluate the sum in the first part of Proposition \ref{prop major}; suppose that $q \leq KQ_0 \leq \sqrt{N}$, 
as in Proposition \ref{prop major}. Using 
$$ 
\int_{-\frac{K}{qQ}}^{\frac{K}{qQ}} {\hat \Phi}(\beta N) e(n\beta) d\beta = \frac{1}{N} \int_{-\frac{NK}{qQ}}^{\frac{NK}{qQ}} {\hat \Phi}(u) e(n u/N) du = \frac{1}{N} \Big(\Phi \Big(\frac nN\Big) + O \Big ( \min\Big(1,\frac{qQ}{KN}\Big)\Big)\Big), 
$$ 
and that 
\begin{align*}
\sum_{n\le N } \Lambda(n) c_q(n)  \Big(\Phi \Big(\frac nN\Big) &+ O \Big ( \min\Big(1,\frac{qQ}{KN}\Big)\Big)\Big)\\
& =  \mu(q) \sum_{n\le N } \Lambda(n) \Big(\Phi \Big(\frac nN\Big) + O \Big ( \min\Big(1,\frac{qQ}{KN}\Big)\Big)\Big) + O\Big(\sum_{p^{k} || q} p^{k} \log N\Big) \\
& =  \mu(q) N \int_0^1 \Phi(t) dt + O\Big(\frac{N}{(\log N)^{10}} + \min \Big( N, \frac{qQ}{K}\Big)\Big),
\end{align*}
we conclude that 
$$ 
\int_{-\frac{K}{qQ}}^{\frac{K}{qQ}} \sum_{n\le N} \Lambda(n) c_q(n)e(n\beta) {\hat \Phi}(N\beta) d\beta = \mu(q)  \int_0^1 \Phi(t) dt + O\Big(\frac{1}{(\log N)^{10}} + \min \Big(1, \frac{qQ}{KN}\Big)\Big). 
$$ 
Using this together with Proposition \ref{prop major} and \eqref{5.4} we obtain, with a small calculation,  
\begin{align*}
\int_{\frak M} \psi(N;\alpha)\tilde{\psi}(N;-\alpha) d\alpha &= N \sum_{q\le KQ_0} \frac{\mu(q)^2}{\phi(q)} \int_0^1 \Phi(t) dt  + O\Big(N+N \sum_{q \leq KQ_0} \frac{1}{\phi(q)} \min\Big(1,\frac{qQ}{KN}\Big) \Big) \\
&=N \sum_{q\le KQ_0} \frac{\mu(q)^2}{\phi(q)} \int_0^1 \Phi(t) dt  + O(N\log \log N)\\
&= N\Big( \int_0^1 \Phi(t) dt \Big) \Big( \log (KQ_0) + O(\log \log N)\Big). 
\end{align*}
Taking the difference between this and \eqref{5.2} we conclude that 
\begin{equation} 
\label{5.5} 
\Big|\int_{\frak m} \psi(N;\alpha) \tilde{\psi}(N;-\alpha) d\alpha \Big| = N \Big( \int_0^1 \Phi(t) dt \Big) \Big( \log \frac{R}{KQ_0} + O(\log \log N)\Big). 
\end{equation} 

Using the second part of Proposition \ref{prop major} and \eqref{5.4}, we similarly get that 
$$ 
\int_{\frak M} |{\tilde \psi}(N;\alpha)|^2 d\alpha = N \Big( \int_0^1 \Phi(t)^2 dt \Big) \log (KQ_0) + O(N\log \log N).
$$  
Thus, using \eqref{5.3}, we conclude that 
 \begin{equation} 
\label{5.6} 
\int_{\frak m} |\tilde{\psi}(N;\alpha)|^2d\alpha = N \Big( \int_0^1 \Phi(t)^2 dt \Big) \Big( \log \frac{R}{KQ_0} + O(\log \log N)\Big). 
\end{equation}

{\em Conclusion.} Combining \eqref{5.5} and \eqref{5.6} with Cauchy--Schwarz (as in \eqref{1.7}), and recalling our choice of $\Phi(t)$ as a smooth approximation from below to the indicator function of $[0,1]$, we obtain 
\begin{equation*} 
\label{5.9} 
\int_{\frak m} |\psi(N;\alpha)|^2 d\alpha \ge N (1 +O(\epsilon))  \Big( \log \frac{R}{KQ_0} + O(\log \log N)\Big), 
\end{equation*} 
which when used with \eqref{5.1} (and the choices of $R, K, Q_0$) yields the theorem. 
 
 \section{Estimates for divisor sums} 
 
\noindent In this section we collect together various estimates for averages of divisor functions, which we will need for our proof of Theorem 2.  Since the proofs of these facts are largely routine applications of contour integration, we will content ourselves with  sketching the proofs  quickly.

\begin{proposition} \label{Prop3} Given a natural number $q$, define 
$$ 
F_q(s) := \prod_{\substack{p^{a} \Vert q \\ a\ge 1 }} \Big(1-\frac 1{p^s}\Big)^k  
\Big( - \frac{d_k(p^{a-1})}{p^{(a-1)(s-1)}} + \phi(p^a) \sum_{b\ge a} \frac{d_k(p^b)}{p^{bs} }\Big). 
$$ 
Then $F_q(s)$ converges absolutely for Re$(s)>0$, and in the region Re$(s)>1$ we have 
$$ 
\sum_{n=1}^{\infty} \frac{d_k(n) c_q(n) }{n^s} = \zeta(s)^k F_q(s). 
$$ 
Uniformly for $q\le N$ we have 
$$ 
\sum_{n\le N} d_k(n) c_q(n)  = \mathop{\text{Res}}_{s=1} \Big( \zeta(s)^k F_q(s) \frac{N^s}{s} \Big) + O_{k,\epsilon}\Big( N^{1+\epsilon} \Big(\frac qN\Big)^{\frac{2}{k+2}}\Big). 
$$ 
\end{proposition} 

\begin{proof}  The first assertion follows upon using  \eqref{1.4} and computing Euler products.   The second assertion follows by a standard contour shift argument, starting with a quantitative Perron formula 
$$ 
\sum_{n\le N} d_k(n) c_q(n) = \frac{1}{2\pi i} \int_{1+1/\log N -iT}^{1+1/\log N +iT} \zeta(s)^k F_q(s) N^{s} \frac{ds}{s} + O_{k,\epsilon}\Big(qN^{\epsilon} + \frac{N^{1+\epsilon}}{T}\Big), 
$$ 
and then moving the line of integration to the line segment from $\epsilon - iT$ to $\epsilon +iT$.  The pole at $s=1$ gives the stated main term.   Using the 
convexity bound $|\zeta(s)|^k \ll (1+|s|)^{k(1-\sigma)/2+\epsilon}$ and the easy bound $|F_q(s)| \ll q^{(1-\sigma)+\epsilon}$, we can bound the other integrals producing an additional error term $O_{k,\epsilon}(N^{\epsilon} qT^{k/2})$.  Finally choosing $T= (N/q)^{\frac 2{k+2}}$, the proposition follows.  
\end{proof}

 \begin{proposition} \label{Prop4}  Given a natural number $q$, and a natural number $k$, define 
 $$ 
 G_{q}(s) := \prod_{p^a \Vert q} \Big(1-\frac{1}{p^s}\Big)^{k-1} \sum_{b=a}^{\infty} \frac{d_{k-1}(p^b)}{p^{bs} } = \frac{1}{q^{s}} \prod_{p^a \Vert q} \Big(1-\frac{1}{p^s}\Big)^{k-1} \sum_{b=a}^{\infty} \frac{d_{k-1}(p^b)}{p^{(b-a)s} } . 
 $$ 
Then $G_{q}(s)$ converges absolutely in the region Re$(s)>0$, and in the region Re$(s)>1$ we have 
$$ 
\sum_{q|n} \frac{d_{k-1}(n)}{n^s} = \zeta(s)^{k-1} G_{q}(s). 
$$ 
Uniformly for $q\le x$ we have 
$$ 
\sum_{\substack{n\le x\\ q|n } } \frac{ d_{k-1}(n)}{n} = \mathop{\text{Res}}_{s=0} \Big( \zeta(s+1)^{k-1} G_{q}(s+1) \frac{x^s}{s}\Big) + O_{k,\epsilon}\Big( \frac{x^{\epsilon}}{q}  \Big(\frac{q}{x}\Big)^{\frac{2}{k+1}}\Big). 
$$ 
The error term above may depend on $k$ and $\epsilon$, but is uniform in $q$.  
\end{proposition}
\begin{proof}  This is proved similarly to Proposition \ref{Prop3}, by comparing Euler products to establish the 
stated identity, and then shifting contours.   
\end{proof} 

 \section{The case of divisor functions: proof of Theorem 2} 
    
\noindent    
Throughout we take $K=(\log N)^{10}$, and $N^{\frac 12+\delta} \le Q \le N$, and $Q_0 = N^{1+\epsilon}/Q$ for some small $\epsilon > 0$ depending on $\delta$.  
Put also $D_k(N;\alpha) := \sum_{n\le N } d_k( n) e(n\alpha)$.  

{\em Applying Proposition \ref{Prop1.3}.} Proposition \ref{Prop3} may be used to show that $\sum_{n\le N} d_k(n) c_d(n) \ll N^{1+\epsilon}$ for all $d\le N$. (See Lemma \ref{ramdkeval} below for an example calculation of the residue in Proposition \ref{Prop3}.)
Hence 
$$ 
\sum_{q\le Q} \frac 1q \sum_{\substack{d|q \\ d>Q_0}} \frac 1{\phi(d)} \Big| \sum_{n\le N} d_k(n) c_d(n) \Big|^2 
\ll N^{2+\epsilon} \sum_{Q_0 < d \leq Q} \frac 1{\phi(d)} \sum_{\substack{ Q_0 <q \le Q \\ d|q} } \frac 1q \ll \frac{N^{2+\epsilon}}{Q_0} \le QN ,
$$ 
 and applying Proposition \ref{Prop1.3},  we get 
\begin{equation} 
\label{6.1} 
\sum_{Q_0 < q\le Q} V_k(q) \ge Q (1+o(1)) \int_{\frak m} |D_k(N;\alpha)|^2 d\alpha + O(QN).
\end{equation}

 
{\em Preparations for Proposition \ref{prop major}.} Now let $R$ be a parameter with $Q_0 N^{\epsilon} \le R \le QN^{-\epsilon}$, and take $b_r = d_{k-1}(r)$ for $r\le R$ and $b_r=0$ for $r>R$.  Set 
 $$ 
 {\tilde d}_k(n) := \sum_{\substack{r|n \\ r\le R} } d_{k-1}(r) \Phi\Big(\frac nN\Big), \qquad \text{and} 
 \qquad 
 {\tilde D}_k(N;\alpha) := \sum_{n \leq N} {\tilde d}_k(n) e(n\alpha). 
 $$ 
 Note that 
 \begin{equation} 
 \label{6.3} 
 \int_{\frak m} |{\tilde D}_k(N;\alpha)|^2 d\alpha \le \int_0^1 |{\tilde D}_k(N;\alpha)|^2  = \sum_{n} {\tilde d}_k(n)^2 \le 
 \sum_{n\le N} d_k(n)^2 \sim c_k N (\log N)^{k^2 -1},  
 \end{equation}  
 where the final asymptotic is a routine calculation, and 
 $$ 
 c_k = \frac{1}{(k^2 -1)!} \prod_p \Big(1-\frac 1p\Big)^{k^2} \Big( \sum_{a=0}^{\infty} \frac{d_k(p^a)^2}{p^{a}}\Big). 
 $$ 
 
We now give two ways to finish the proof of Theorem \ref{kdiv}.  Our first approach, carried out in Section 8.1, establishes 
that for some choice of $R$ in $[Q_0N^{\epsilon},QN^{-\epsilon}]$ we have 
 \begin{equation} 
 \label{6.4} 
 \Big| \int_{\frak m} D_k(N;\alpha) \overline{{\tilde D}_k(N;\alpha)} d\alpha 
 \Big | \gg_{k,\delta} N (\log N)^{k^2 -1}, 
 \end{equation} 
 so that  combining \eqref{6.1} with \eqref{6.3} and \eqref{6.4} (together with Cauchy--Schwarz as in \eqref{1.7}) we would  deduce the theorem. The second approach, carried out in Section 8.2, establishes that with $R = N^{\frac 12-\frac{\delta}{2}}$ one has 
\begin{equation} 
 \label{absvalueseasier} 
 \int_{\frak m} |D_k(N;\alpha) {\tilde D}_k(N;\alpha)| d\alpha 
\gg_{k,\delta} N (\log N)^{k^2 -1},
 \end{equation}
 which again by Cauchy--Schwarz (as in \eqref{1.13}) is enough to deduce the theorem. 
 
\subsection{Proof of Theorem 2: the first ending}
 We begin our proof of \eqref{6.4} by noting that, using Proposition 
 \ref{prop major} 
 \begin{align} 
 \label{6.5} 
 \int_{\frak M} D_k(N;\alpha) \overline{{\tilde D}_k(N;\alpha)} d\alpha & = N \sum_{q\le KQ_0} \Big(\sum_{\substack {r\le R\\ q|r}} \frac{d_{k-1}(r)}{r}\Big) \sum_{n\le N} d_k(n)c_q(n) \int_{-\frac{K}{qQ}}^{\frac K{qQ}} e(n\beta) {\hat \Phi}(\beta N) d\beta \nonumber\\
 &\hskip 1 in  + O(N^{1+\epsilon}R/Q).
 \end{align} 
 
 \begin{lemma} \label{lem7.1}  We have 
 \begin{align*}
 \sum_{q\le KQ_0}  \Big(\sum_{\substack {r\le R\\ q|r}} &\frac{d_{k-1}(r)}{r}\Big) \sum_{n\le N} d_k(n)c_q(n) \int_{-\frac{K}{qQ}}^{\frac K{qQ}} e(n\beta) {\hat \Phi}(\beta N) d\beta \\
 &= 
\frac{1}{N} \sum_{q\le KQ_0}  \Big(\sum_{\substack {r\le R\\ q|r}} \frac{d_{k-1}(r)}{r}\Big) \sum_{n\le N} d_k(n)c_q(n) \Phi\Big(\frac nN\Big) +  O\Big((\log N)^{k^2-2} \log \frac{Q_0Q}{N}\Big). 
 \end{align*}
 \end{lemma} 
 \begin{proof}  Since $\int_{-\infty}^{\infty}e(n\beta) {\hat \Phi}(\beta N) d\beta = N^{-1}\Phi(n/N)$, it is sufficient to estimate 
 \begin{equation} 
 \label{7.1.1} 
 \sum_{q\le KQ_0}  \Big(\sum_{\substack {r\le R\\ q|r}} \frac{d_{k-1}(r)}{r}\Big) \sum_{n\le N} d_k(n)c_q(n) 
 \int_{|\beta|> \frac{K}{qQ}} e(n\beta) {\hat \Phi}(\beta N) d\beta .
 \end{equation} 
 Now, by partial summation and using Proposition \ref{Prop3} we may see (recalling here that $q \leq KQ_0 \leq \sqrt{N}$ and $k \geq 2$, so $d_{k-1}(q) \geq d_{1}(q) = 1$) that 
 $$ 
 \sum_{n\le N} d_{k}(n) c_q(n) e(n\beta) \ll (1+N|\beta|) d_{k-1}(q) \prod_{p|q} \Big(1+O_{k}\Big(\frac 1p\Big)\Big) N(\log N)^{k-1}. 
 $$ 
 Using this, and the straightforward estimates 
 $$
 \sum_{\substack {r\le R\\ q|r}} \frac{d_{k-1}(r)}{r} \leq \frac{d_{k-1}(q)}{q} \sum_{s \leq R/q} \frac{d_{k-1}(s)}{s} \ll_{k} \frac{d_{k-1}(q)}{q} (\log R)^{k-1},
 $$
  the quantity in  \eqref{7.1.1} may be bounded by  
 $$ 
 \ll N(\log N)^{k-1} \sum_{q\le KQ_0} \frac{d_{k-1}(q)^2}{q} \prod_{p|q} \Big(1+O_{k}\Big(\frac 1p\Big)\Big) (\log R)^{k-1} \int_{|\beta|> \frac{K}{qQ}} (1+|\beta|N) |{\hat \Phi}(\beta N)| d\beta,
 $$
 which is 
\begin{align*} 
& \ll (\log N)^{2k-2} \sum_{q\le KQ_0} \frac{d_{k-1}(q)^2}{q} \prod_{p|q} \Big(1+O_{k}\Big(\frac 1p\Big)\Big)  \min\Big(1,\frac{qQ}{KN}\Big)\\
&  \ll_{k} (\log N)^{2k-2} (\log N)^{(k-1)^2- 1} \log \frac{Q_0Q}{N}.
\end{align*}
 \end{proof}

Now, as in \eqref{1.8},  
 \begin{equation*} 
 \int_{\frak m} D_k(N;\alpha) \overline{{\tilde D}_k(N;\alpha)} d\alpha = \sum_n d_k(n) {\tilde d}_k(n) - \int_{\frak M} D_k(N;\alpha) \overline{{\tilde D}_k(N;\alpha)} d\alpha , 
 \end{equation*} 
 and combining this with Lemma \ref{lem5} and Lemma \ref{lem7.1} and \eqref{6.5}, we obtain (up to an error term of $O(N(\log N)^{k^2-2} \log(Q_0 Q/N))$) 
 \begin{equation}
 \label{6.6} 
  \int_{\frak m} D_k(N;\alpha) \overline{{\tilde D}_k(N;\alpha)} d\alpha  \sim  \sum_{ KQ_0<q \le R}  \Big(\sum_{\substack {r\le R\\ q|r}} \frac{d_{k-1}(r)}{r}\Big) \sum_{n\le N} d_k(n)c_q(n) \Phi\Big(\frac nN\Big) .
  \end{equation} 
Notice here that $\log(Q_0 Q/N) = \epsilon \log N$, so the error term will be small compared with $N (\log N)^{k^2-1}$ provided we ultimately choose $\epsilon$ small enough. We may now use our work in Propositions \ref{Prop3} and \ref{Prop4} to evaluate the sums over $n$ and $r$ above.  Thus, 
 with an error term of at most $o(N(\log N)^{k^2-1})$, the right hand side of \eqref{6.6} equals 
 \begin{equation} 
 \label{6.7} 
 \sum_{KQ_0 <q\le R} \Big( \mathop{\text{Res}}_{s=0} \zeta(s+1)^{k-1} G_{q}(s+1) \frac{R^s}{s}\Big) \Big( \mathop{\text{Res}}_{w=1} \zeta(w)^k F_q(w) \Big(\int_0^1 \Phi(y)y^{w-1} dy \Big) N^w 
 \Big). 
 \end{equation} 
 
 In \eqref{6.7}, we think of the residues over $s$ and $w$ as contour integrals over circles centered around $s=0$ and $w=1$ with 
 radius $1/\log N$.  With this range for $s$ and $w$ in mind, we consider the sum over $q$, which we may itself write as a contour integral as (assuming $R, KQ_0$ are not integers)
 $$ 
 \sum_{KQ_0<q\le R} F_q(w)G_{q}(s+1) = \frac{1}{2\pi i }\int_{(c)} \sum_{q=1}^{\infty} \frac{F_q(w)G_{q}(s+1)}{q^z} \frac{R^z-(KQ_0)^z}{z} dz, 
 $$
  where the integral is taken over the line Re$(z)=c$ with $c=10/\log N$, say.  Now $F_q(w) G_{q}(s+1)$ is a multiplicative function of $q$, and a little calculation shows that in the $p$-factor in the corresponding Euler product, the leading terms are $1 - \frac{(k-1)}{p^{z+s+1}} + \frac{k(k-1)}{p^{z+s+w}}$ (the next terms all involving a larger multiple of $z+s+w$, of $s+1$, or of $w$ in the exponent of $p$). So it turns out that we may write
  $$ 
  \sum_{q=1}^{\infty} \frac{F_q(w)G_{q}(s+1)}{q^z} = \frac{\zeta(s+w+z)^{k(k-1)}}{ \zeta(s+1+z)^{k-1}} H(z;s,w), 
  $$ 
  where (for any $s,w$ in our small discs) $H(z;s,w)$ is analytic in Re$(z)>-1/2+\epsilon$ for any $\epsilon>0$, and bounded in that region.  Using this in \eqref{6.7} and moving the line of integration over $z$ to the left, we can write the quantity we need to compute in \eqref{6.7} (up to an acceptable error $o(N(\log N)^{k^2-1})$) as 
  \begin{align*}
  \mathop{\text{Res}}_{z=0} \mathop{\text{Res}}_{s=0} \mathop{\text{Res}}_{w=1} \zeta(s+1)^{k-1} \zeta(w)^k \frac{\zeta(s+w+z)^{k(k-1)}}{\zeta(s+1+z)^{k-1}}& H(z;s,w)\Big(\int_0^1 \Phi(y)y^{w-1}dy \Big)\\
  &\times \frac{R^s N^w (R^z-(KQ_0)^z)}{sz}. 
  \end{align*}
  Now, in computing the residues above, we may replace terms $\zeta(u)$ for $u$ near $1$ by $1/(u-1)$, and also replace $H(z;s,w)$ by $H(0;0,1)$ and $\int_0^1 \Phi(y)y^{w-1} dy$ by $\int_0^1 \Phi(y)dy$.  These changes affect the residue above only to order $N(\log N)^{k^2-2}$.  Thus, our desired main term is (also replacing $w$ by $w+1$)  
  \begin{equation} 
  \label{6.8}
  NH(0;0,1) \Big(\int_0^1 \Phi(y) dy \Big)\mathop{\text{Res}}_{z=0} \mathop{\text{Res}}_{s,w=0} \frac{1}{s^{k-1} w^k} \frac{(s+z)^{k-1}}{(s+z+w)^{k(k-1)}} \frac{R^s N^w (R^z-(KQ_0)^z)}{sz}.
  \end{equation}
  
 A straightforward calculation shows that 
 $$ 
 H(0;0,1) = \prod_p \Big(1-\frac 1p\Big)^{k^2} \Big(\sum_{a=0}^{\infty} \frac{d_k(p^a)^2}{p^a} \Big), 
 $$ 
 matching the natural Euler factor that arises in the asymptotic for $\sum_{n\le N} d_k(n)^2$.  Further $\int_0^{1} \Phi(y) dy =1+O(\epsilon)$ by our choice of $\Phi$.  
 Finally, another straightforward calculation gives that the residues in \eqref{6.8} equal 
 \begin{align} 
 \label{6.9} 
 \sum_{\ell=0}^{k-1} \sum_{j=0}^{k-1} \binom{-k(k-1)}{\ell} \binom{-(k-1)^2-\ell}{j} &\frac{(\log N)^{k-1-\ell}}{(k-1-\ell)!} \frac{(\log R)^{k-1-j}}{(k-1-j)!} \nonumber \\ 
 &\times \frac{(\log R)^{(k-1)^2+\ell+j} - (\log KQ_0)^{(k-1)^2+\ell+j}}{((k-1)^2+\ell+j)!}.
 \end{align} 
 In performing this calculation, it is helpful to write 
 $$
\frac{ 1}{(s+z+w)^{k(k-1)}} = (s+z)^{-k(k-1)} \Big(1 + \sum_{\ell=1}^{\infty} \binom{-k(k-1)}{\ell} \Big(\frac{w}{s+z}\Big)^{\ell}\Big )
 $$
 to compute first the residue in $w$, 
 and then write
  $$
\frac{1}{  (s+z)^{(k-1)^2 + \ell}} = z^{-(k-1)^2 - \ell}\Big (1 + \sum_{j=1}^{\infty} \binom{-(k-1)^2 - \ell}{j} \Big(\frac{s}{z}\Big)^{j} \Big)
  $$ 
  to compute the residue in $s$, and then compute the residue in $z$ as the final step.  
 
   At this stage, we have successfully evaluated our desired quantity \eqref{6.6}.  However, it is not immediately clear that the quantity in \eqref{6.9}, which is clearly $\ll_{k} (\log N)^{k^2-1}$, cannot somehow cancel out to zero.   In our argument we have so far left $R$ to be an unspecified non-integer value lying between $Q_0N^{\epsilon}$ 
   and $QN^{-\epsilon}$.  We may expect that the expression in \eqref{6.9} is positive and increasing in $R$ in that range, so that the optimal choice for $R$ would be $QN^{-\epsilon}$.   But it does not seem straightforward to establish that claim, assuming it is true!  Instead we may circumvent this difficulty as follows. After scaling by $(\log N)^{k^2-1}$, the expression in \eqref{6.9} is, for fixed $N$, $KQ_0$,  a polynomial in $\alpha =\log R/\log N$ of degree $k^2-1$.  The leading coefficient of this polynomial can be readily calculated: it equals the $\ell=k-1$ term, namely
 \begin{align*}
 \binom{-k(k-1)}{k-1} &\sum_{j=0}^{k-1} \binom{-k(k-1)}{j} \frac{1}{(k-1-j)!} \frac{1}{(k(k-1)+j)!} \\
&= (-1)^{k-1} \frac{1}{(k-1)! (k(k-1)-1)! (k^2-1)}.
\end{align*}  
 Recall that over a given interval, any polynomial of a given degree and 
 leading coefficient attains in size a value that may be bounded below just in terms of the degree, the leading coefficient, and the length 
 of the interval.  Indeed, scaled and translated versions of the Chebyshev polynomials minimize this maximal size.  Since $\alpha$ is allowed to vary in an interval 
 of length $(\log (Q^2/N)/\log N-\epsilon) \gg_{\delta} 1$, we conclude that for some $R$ in $[Q_0N^{\epsilon},QN^{-\epsilon}]$, our quantity \eqref{6.8} has size $\ge C (\log N)^{k^2-1}$ for some 
 constant $C$ depending only on $k$ and $\delta \le  \log (Q/\sqrt{N})/\log N$.   This completes our proof.

\subsection{Proof of Theorem 2: the second ending}
To finish the paper we offer a different ending to the proof of Theorem \ref{kdiv}, by proving \eqref{absvalueseasier}.  
Recall that $\delta>0$ is suitably small, $Q\ge N^{\frac 12+\delta}$ and that $Q_0 = N^{1+\epsilon}/Q \le N^{\frac 12-\delta+\epsilon}$.  We take $R=N^{\frac 12-\frac{\delta}{2}}$, and appeal to Proposition \ref{newprop}.  We conclude that  
$$ 
\int_{\frak m} |D_k(N;\alpha) {\tilde D}_k(N;\alpha)| d\alpha 
\ge \sum_{KQ_0 < q\le R}  \sum_{\substack{r\le R \\ q|r}} \frac{d_{k-1}(r)}{r} 
\Big| \sum_{n\le N} d_k(n) c_q(n) \Phi\Big(\frac{n}{N}\Big) \Big| + O(N^{1-\epsilon}). 
$$ 
Now for any $q\le R$, note that 
\begin{equation*}
\sum_{\substack{r \leq R \\ q|r}} \frac{d_{k-1}(r)}{r}
 \geq  \frac{d_{k-1}(q)}{q} \sum_{\substack{m \leq R/q \\ (m,q)=1}} \frac{d_{k-1}(m)}{m} 
 \geq  \frac{d_{k-1}(q)}{q} \prod_{p|q} \left(1- \frac{1}{p} \right)^{k-1} \sum_{m \leq R/q} \frac{d_{k-1}(m)}{m} . 
\end{equation*}
A standard calculation shows that the above is 
$$ 
\gg_{k} \frac{d_{k-1}(q)}{q} \Big(\frac{\phi(q)}{q} \Big)^{k-1} (\log (R/q))^{k-1}, 
$$ 
and therefore 
\begin{equation} 
\label{altpr1}
\int_{\frak m} |D_k(N;\alpha) {\tilde D}_k(N;\alpha)| d\alpha  \gg \sum_{KQ_0 < q\le R} 
\frac{d_{k-1}(q)}{q}\Big(\frac{\phi(q)}{q} \log \frac{R}{q}\Big)^{k-1} \Big| 
\sum_{n\le N} d_k(n) c_q(n) \Phi\Big(\frac{n}{N}\Big) \Big|. 
\end{equation}

In Proposition \ref{Prop3} we saw how to evaluate the sum over $n$ in \eqref{altpr1} as a residue, but 
that residue calculation can be complicated, as we saw in the previous section.  Now we show that for certain 
values of $q$, one may obtain a lower bound for this residue and this will be enough to deduce our desired lower 
bound \eqref{absvalueseasier}.

\begin{lemma}\label{ramdkeval}
For any natural number $k \geq 2$, and any small $\delta > 0$, there exists a small constant $c_{k,\delta} > 0$ such that the following is true. If $N$ is large enough depending on $k$ and $\delta$, and if $q \leq N^{\frac 12 -\frac{ \delta}2}$ is squarefree and composed only of primes below $N^{c_{k,\delta}}$, then
$$
 \Big| \sum_{n \leq N} d_k(n) \Phi(n/N) c_q(n) \Big| 
 \gg_{k, \delta} d_{k-1}(q) \Big( \frac{\phi(q)}{q}\Big)^{k} N  (\log N)^{k-1}. 
 $$
\end{lemma}

Assuming the lemma, we can quickly finish our second proof of Theorem \ref{kdiv}.  Restricting attention to $KQ_0 < q\le RN^{-\delta/4}$ 
with $q$ square-free and composed only of primes below $N^{c_{k,\delta}}$, the sum in \eqref{altpr1} is (with the $\star$ on the sum indicating these conditions)
$$ 
\gg_{k,\delta} N(\log N)^{2(k-1)} \sum_{q}^{\star} \frac{d_{k-1}(q)^2} {q} \Big(\frac{\phi(q)}{q}\Big)^{2k-1}.  
$$
Now one can show that
$$ 
\prod_{N^{c_{k,\delta}} \le p\le R} \Big(1-\frac{1}{p}\Big)^{-(k-1)^2} 
\sum_q^{\star} \frac{d_{k-1}(q)^2}{q} \Big(\frac{\phi(q)}{q}\Big)^{2k-1} \gg_{k,\delta} \sum_{KQ_0 < q\le RN^{-\delta/4}}^{\prime} \frac{d_{k-1}(q)^2}{q}  \Big( \frac{\phi(q)}{q}\Big)^{2k-1},  
$$ 
 where the $\prime$ indicates that the smoothness condition on $q$ has been removed, but the square-free condition kept in place. (If the sum had $q \leq R N^{-\delta/4}$, rather than $KQ_0 < q \leq R N^{-\delta/4}$, this would follow trivially as in the manipulations at the beginning of this subsection. To deal with the interval condition, one can compare the numbers $q$ appearing in different intervals of multiplicative length $N^{\delta/5}$ to show that the sum over each interval is of the same order of magnitude.) Then either by elementary arguments, or through a straightforward contour shift argument we may see that the sum above is $\gg_{k,\delta} (\log N)^{(k-1)^2}$.  It follows that our quantity in \eqref{altpr1} is 
 $\gg_{k,\delta} N(\log N)^{k^2- 1}$, which establishes \eqref{absvalueseasier}.  
 
\begin{proof}[Proof of Lemma \ref{ramdkeval}] We use our work from Proposition \ref{Prop3}.  Our goal will be to show that for $q$ as in the lemma, one has 
\begin{equation} 
\label{altpr2}
\mathop{\text{Res}}_{s=1} \left(\zeta(s)^{k} \frac{F_{q}(s)}{F_{q}(1)} \frac{N^{s-1}}{s} \right) \gg_{k,\delta} (\log N)^{k-1}. 
\end{equation} 
Observe that for square-free $q$, the definition of $F_q(s)$ may be simplified:  
$$ 
F_q(s) = \prod_{p|q} \Big(1-\frac{1}{p^s}\Big)^k \Big( -1 + \phi(p) \sum_{b=1}^{\infty}\frac{d_k(p^b)}{p^{bs}} \Big) 
= \prod_{p|q} \Big( p-1 - p \Big(1-\frac{1}{p^s}\Big)^k \Big). 
$$ 
Thus, in particular, 
$$ 
F_q(1) = \prod_{p|q}  \Big(1-\frac 1p\Big)^{k} p \Big( \Big(1-\frac{1}{p}\Big)^{-(k-1)} - 1\Big) = \Big(\frac{\phi(q)}{q}\Big)^{k} \prod_{p|q} p \Big( \Big(1-\frac{1}{p}\Big)^{-(k-1)} - 1\Big)  \ge d_{k-1}(q) \Big(\frac{\phi(q)}{q}\Big)^{k}, 
$$
 and so the lemma will follow from \eqref{altpr2} and partial summation.  
 

To estimate the residue in \eqref{altpr2} it is helpful to let $f_q(s)$ denote the logarithmic derivative $F_{q}'(s)/F_{q}(s)$, 
so that by Taylor's theorem we have (in a neighbourhood of $s=1$)
$$ 
\frac{F_q(s)}{F_q(1)} = \exp\{\log F_q(s) - \log F_q(1)\} = \exp\Big((s-1) f_{q}(1) + \frac{(s-1)^2}{2!} f_{q}'(1) + \frac{(s-1)^3}{3!} f_{q}''(1) + ... \Big) .
 $$
 A quick calculation gives 
 \begin{equation} 
 \label{altpr3} 
 f_q(s)= - \sum_{p|q} \frac{k\log p}{p^{s-1}} \Big\{ (p-1) \Big(1-\frac{1}{p^s}\Big)^{-(k-1)} - p\Big(1-\frac 1{p^s}\Big)\Big\}^{-1} , 
 \end{equation} 
and in particular 
\begin{equation} 
\label{altpr4} 
0\ge f_q(1) = - \sum_{p|q } \frac{k\log p}{p-1} \Big\{ \Big(1-\frac 1p \Big)^{-(k-1)} -1\Big\}^{-1} \ge - \frac{k}{k-1} \sum_{p|q} \frac{p}{p-1} \log p. 
\end{equation}  
Further repeated differentiation shows that for any non-negative integer $0\le m\le k$ one has, for a suitable constant $C(k)$   
\begin{equation} 
\label{altpr5} 
|f_q^{(m)}(1)| \le C(k) \sum_{p|q} (\log p)^{m+1} \le C(k) (\log q) \Big( \max_{p|q} (\log p)^m \Big). 
\end{equation}  

With these calculations in place, we return to the residue in \eqref{altpr2}, which is the coefficient of $1/(s-1)$ in the Laurent expansion around $s=1$ 
of 
$$ 
\frac{1}{(s-1)^k} \frac{((s-1)\zeta(s))^k}{s} \exp\Big( (\log N +f_q(1)) (s-1) + \frac{(s-1)^2}{2!} f_1^{\prime}(1) + \frac{(s-1)^3}{3!} f_q^{\prime\prime}(1)+ \ldots \Big),  
$$ 
which equals 
$$ 
\sum_{j=0}^{k-1} \frac{(\log N+f_q(1))^{k-1-j}}{(k-1-j)!} \mathop{\text {Res}}_{s=1} \Big( \frac{((s-1) \zeta(s))^k}{(s-1)^{j+1} s} \exp\Big( \frac{(s-1)^2}{2!} f_q^{\prime}(1) +\ldots \Big) \Big). 
$$ 
Using \eqref{altpr5} we may see that the terms $j\ge 1$ above contribute an amount that is
$$ 
\le C(k) (\log N)^{k-2} \Big( \max_{p|q} \log p\Big), 
$$
for a suitable (different) constant $C(k)$.  Therefore the residue we seek is 
$$ = \frac{(\log N+f_q(1))^{k-1}}{(k-1)!} + O_{k}\left( (\log N)^{k-2} \Big( \max_{p|q} \log p\Big) \right) = \frac{(\log N+f_q(1))^{k-1}}{(k-1)!} + O_{k}\left( c_{k,\delta} (\log N)^{k-1} \right) , $$
where $c_{k,\delta}$ is as in the statement of the lemma. In particular, using \eqref{altpr4} we obtain \eqref{altpr2} provided $q\le N^{\frac 12-\frac{\delta}{2}}$ and provided $c_{k,\delta}$ is small enough.  
\end{proof}

 \bibliographystyle{plain} 
 \bibliography{RefsUni}{}

  \end{document}